\documentclass[12pt,a4paper,reqno]{amsart}
\usepackage[latin1]{inputenc}
\usepackage[T1]{fontenc}
\usepackage{lmodern}
 \usepackage[english]{babel}
\usepackage{amsmath}
\usepackage{hyperref}
\usepackage{amsfonts}
\usepackage{amssymb}

\numberwithin{equation}{section}\theoremstyle{plain}
\newtheorem{theorem}{Theorem}[section]
\newtheorem{proposition}[theorem]{Proposition}
\newtheorem{lemma}[theorem]{Lemma}

\newtheorem{definition}[theorem]{Definition}
\newtheorem{remark}[theorem]{Remark}

\title[]
{Wigner-Ville distribution associated with\\ the quaternion offset linear canonical transforms}

\address{\begin{center}{\small Department of Mathematics and Computer Sciences, Faculty of Sciences,\\
Equipe d'Analyse Harmonique et Probabilit\'{e}s, University Moulay Isma\"{\i}l,\\
BP 11201 Zitoune, Meknes, Morocco}
\end{center}}

\begin{document}

\author[ M. El kassimi, Y. El haoui, and S. Fahlaoui ]{ Mohammed El kassimi ,  Youssef El haoui   and Sa\"{\i}d Fahlaoui  }

\address{Mohammed El kassimi} \email{m.elkassimi@edu.umi.ac.ma}

\address{Youssef El haoui} \email{y.elhaoui@edu.umi.ac.ma}

\address{Sa\"{\i}d Fahlaoui}  \email{s.fahlaoui@fs.umi.ac.ma}
\begin{abstract}
The  Wigner-Ville distribution (WVD) and quaternion  offset linear canonical transform (QOLCT) are a useful tools in signal analysis and image processing. The purpose of this paper is to define the Wigner-Ville distribution associated with quaternionic offset linear canonical transform (WVD-QOLCT). Actually, this transform combines both the results and flexibility  of the two transform WVD and QOLCT. We derive some important properties of this transform such as inversion and Plancherel formulas,  we  establish a version of Heisenberg inequality, Lieb's theorem and we give the Poisson summation formula for the WVD-QOLCT.
\end{abstract}
\maketitle
{\it keywords:} Wigner-Ville distribution, Offset linear canonical transform, linear canonical transform, quaternionic transform,Heisenberg uncertainty.\\

\section{Introduction}
The Fourier transformation used for a simple description of the input-output relationships of the filters linear, occupies a privileged place in the theory
and signal processing. However, this transformation can not give a temporal signal, it only gives a global frequency information: its natural field of application is analysis stationary signals. So, as soon as we consider modulated signals or non-process
stationary the Fourier transform becomes insufficient to study this type of signal. One solution to this problem is to associating to directly search a
tool adapted to the study of non-stationary signal, without direct reference to the methods resulting from the stationary case. In this case, a particular axis
of interest has been manifested for many years to a proposed transformation in Quantum Mechanics by E. P. Wigner \cite{WIG} in 1932. This transformation allows to define what we will call the distribution of Wigner-Ville (WVD) in reference and tribute to J. City which first introduced this same notion in Signal Theory. In recent years, this distribution has served as a useful analysis tool in many fields as diverse as optics, biomedical engineering, signal processing and image processing. Due to the large applications of the linear canonical transform (LCT)\cite{Xu} in several area including radar analysis, signal processing and optics \cite{Ozaktas,Pei1,Tao}. The LCT  has received attention since 1970 is introduced integral transform with four parameters (a,b,c,d) \cite{colli}\cite{Moshi}. A lot of authors were  interested to study LCT. This transform is also known under the affine Fourier transform \cite{Abe}, and the generalized Fresnel Fourier transform \cite{James}. Moreover the Fourier transform \cite{Brac} and the Fresnel transform \cite{Godman} are all special cases of the LCT. In \cite{Pei1}, the LCT is generalized by introducing two extra parameters, one corresponding to time shift and an other to frequency modulation. This generalized of LCT is called offset LCT (OLCT)\cite{Stern,Zhi}, and it is known under six parameters linear transform. These two parameters make the OLCT more general and flexible than LCT, in consequence the OLCT can apply to most electrical and optical signal systems. The two-sided quaternionic Fourier transform (QFT) was introduced in \cite{Ell}. The QFT has many application in large domains, in \cite{Ell} the QFT used in analysis of 2D linear time invariant dynamic systems, In \cite{Bas} the authors used the QFT to design a digital color image water marking scheme, in \cite{Viksas} the QFT is used for filtering color images. \\
The main objective of this work is the combination between the WVD, QFT and the OLCT, in order to get the Quaternion Offset Wigner-Ville distribution associated to linear canonical transforms (WVD-QOLCT). The paper is organized as follows, in section 2, we recall the main results about the quaternion algebra and  harmonic analysis related to QFT, QLCT and QOLCT. In section 3, we introduce the WVD-QOLCT , and establish its important properties. The section 4 is devoted to give the analogue of Heisenberg's inequality, Poisson summation formula, and Lieb's theorem for the WVD-QOLCT. In section 5, we
conclude this paper.
\section{Preliminaries}
\subsection{The quaternion algebra}\leavevmode\par
In the present section we collect some basic facts about quaternions, which will be needed throughout the paper.
For all what follows, let $\mathbb{H}$ be the Hamiltonian skew field of quaternions:
$$\mathbb{H}=\{q=q_0+iq_1+jq_2+kq_3;\ q_0, q_1, q_2, q_3 \in \mathbb{R}\},$$
which is an associative noncommutative four-dimensional algebra.\\
where the elements ${  i},{  \ }{  j},{  \ }{  k}$ satisfy the Hamilton's multiplication rules:

$$ij = -ji = k;\ jk = -kj = i;\ ki = -ik = j;\ i^2 = j^2 = k^2=-1.$$
In this way the Quaternionic algebra can be seen as an extension of the complex field $\mathbb{C}$.

Quaternions are isomorphic to the Clifford algebra ${Cl}_{(0,2)}$ of ${\mathbb R}^{(0,2)}$:
\begin{equation}\label{equi1}
\mathbb{H} \cong {Cl}_{(0,2)}.
\end{equation}
The scalar part of  a quaternion  $q \in \mathbb{H}$\ is $q_0$ denoted by $Sc(q)$, the non scalar part(or pure quaternion) of $q$\ is $iq_1+jq_2+kq_3$ denoted by $Vec(q)$.

The quaternion conjugate of $q \in \mathbb{H}$, given by

$$\overline{q}=q_0-iq_1-jq_2-kq_3,$$  is an anti-involution, namely,

$$\overline{qp}= \overline{p} \ \overline{q},\ \overline{p+q}= \overline{p}+\overline{q},\ \overline{\overline{p}}=p.$$

The norm or modulus of $q \in {\mathbb H}$\ is defined by
$${|q|}_Q=\sqrt{q\overline{q}}=\sqrt{{q_0}^2+\ {q_1}^2+{q_2}^2+{q_3}^2}.$$

Then, we have

$${|pq|}_Q={|p|}_Q{|q|}_Q.$$
In particular, when $q=q_0$ is a real number, the module ${|q|}_Q$ reduces to the ordinary Euclidean module $\left|q\right|=\sqrt{{q_0}^2}$.

It is easy to verify that  $0 \ne q \in \mathbb{H}$ implies~:

$$q^{-1}=\frac{\overline{q}}{{{|q|}_Q}^2}.$$

Any quaternion  $q$ can be written as $q$=\ ${{|q|}_Qe}^{\mu \theta },$  where $e^{\mu \theta }$ is understood in accordance with Euler's formula\\
 $e^{\mu \theta }={\cos  \left(\theta \right)\ }+\mu \ {\sin  \left(\theta \right),\ }$ where
$\theta = artan \frac{\left|Vec\left(q\right)\right|_Q}{Sc\left(q\right)}$,\ 0$\le \theta \le \pi $ and \ $\mu $ :=\ $\frac{Vec\left(q\right)}{\left|{  Vec}\left({  q}\right)\right|_Q}$ verifying ${\mu }^2 =\ -1$.\\
Let $\lambda$\ be a pure unit quaternion, $\lambda^2=-1,$ \ clearly, we have for all $x \in {\mathbb R}^2,$
 \begin{equation}\label{normEuler}
{ |e^{\lambda x}|_Q=1.}
\end{equation}
In this paper, we will study the  quaternion-valued signal $f:{\mathbb{R}}^2\to \mathbb{H} $, $f$ which can be expressed as $$f=f_0+i f_1+jf_2+kf_3,$$ with $f_m~:  {\mathbb R}^2 \to \ {\mathbb R}\ for\ m=0,1,2,3.$
Let us introduce the canonical inner product for quaternion valued functions  $f,g\ :{\mathbb R}^2\ \to {\mathbb H}$, as follows:
\begin{equation}\label{inner_product}
<f,g> = \int_{{\mathbb R}^2} {f\left(t\right)\overline{g\left(t\right)}}dt,\   dt={dt}_1{dt}_2.
\end{equation}
Hence, the natural norm is given  by
$${\left|f\right|}_{2,Q}=\sqrt{<f,f>} ={(\int_{{\mathbb R}^2}{{\left|f(t)\right|}^2_Q}dt)}^{\frac{1}{2}},$$

and the quaternion module $L^2({\mathbb R}^2,\ {\mathbb H})$, is given by

$$L^2({\mathbb R}^2,\ {\mathbb H}) = \{f : {\mathbb {\mathbb R}^2\ \to {\mathbb H},\ {\left|f\right|}_{2,Q}< \infty }\}.$$

Furthermore, for $2<p<\infty,$ we introduce the quaternion modules $L^p({\mathbb R}^2, \mathbb H),$\ as
$$L^p({\mathbb R}^2, \mathbb H) = \{f : {\mathbb R}^2\ \to {\mathbb H},\ {\left|f\right|}^p_{p,Q}=\int_{\mathbb{R}^2}|f(x)|^p_Q dx < \infty \}.$$

From \eqref{inner_product}, we obtain the quaternion Schwartz's inequality

\[\forall f,g\in L^2\left({\mathbb R}^2,\mathbb H\right):\ \ \ \ \ {\left|\int_{{\mathbb R}^2}{f(x)}\overline{g(x)}dx\right|^2_Q}\le \int_{\mathbb R^2}{{\left|f(x)\right|_Q^2} dx}\int_{\mathbb R^2}{{\left|g(x)\right|_Q^2}dx}.\]

Besides the quaternion units $i, j, k$, we will use the following real vector notation:\\
$t=(t_1,t_2) \in {\mathbb R}^2, \ |t|^2= {t_1}^2+{t_2}^2, \ f(t) = f(t_1,t_2),\ dt={dt}_1{dt}_2.$\\

\subsection{The general two-sided quaternion Fourier transform }\leavevmode\par

In this subsection, we begin by defining the two-sided QFT, and reminder some properties for this transform,

Let us define the two-sided QFT and provide some properties used in the sequel.

\begin{definition}
[\cite{Hitzer1}]\leavevmode\par
Let\ $\lambda, \mu \in  {\mathbb H}$, be any two pure unit quaternions, i.e., ${\lambda }^2= {\mu }^2=-1.$\\
For  $f$ in    $L^1\left({\mathbb R}^2,{\mathbb H}\right)$,\ the two-sided QFT with respect to $\lambda ; \mu $ is\\
\begin{equation}\label{QFT}
{\mathcal F}^{\lambda, \mu }\{f\}(u) =\int_{{\mathbb R}^2}{e^{-\lambda {{ u}}_1t_1}}\ f(t)\ e^{-\mu {{ u}}_2t_2}dt,~~ where ~t, u \in {\mathbb R}^2.
\end{equation}

\end{definition}

We define a new module of $\mathcal F\{f\}^{\lambda, \mu } $  as follows :

\begin{equation}
{\left\|{\mathcal F}^{\lambda, \mu }\left\{f\right\}\right\|}_Q\ := \sqrt{\sum^{m=3}_{m=0}{{\left|{\mathcal F}^{\lambda, \mu }\left\{f_m\right\}\right|}^2_Q}}.
\end{equation}

Furthermore, we define a new $L^2$-norm of ${\mathcal F\{f\}}$ as follows :\\
\begin{equation}
 {\left\|{\mathcal F}^{\lambda, \mu }\{f\}\right\|}_{2,Q}:=\sqrt{\int_{\mathbb {\mathbb R}^2}{{\left\|{\mathcal F}^{\lambda, \mu }\left\{f\right\}(y)\right\|}^2_Qdy}}.
\end{equation}
It is interesting to observe that ${\left\|{\mathcal  F}^{\lambda,\mu }\{f\}\right\|}_{Q}$    is not equivalent to$\ {\left|{\mathcal  F}^{\lambda,\mu }\{f\}\right|}_{Q}$  unless  $f$ is real valued.\\

\begin{lemma}[Dilation property] $($see  page 50 in \cite{Bulow}$)$  \\
Let $k_1, k_2$ be a positive scalar constants,  we have

\begin{equation}
{\mathcal F}^{\lambda,\mu}\left\{f(t_1,t_2)\right\}\left(\frac{u_1}{k_1},\frac{u_2}{k_1}\right)=k_1k_2{\mathcal F}^{\lambda,\mu}\left\{f(k_1t_1,k_2t_2)\right\}\left(u_1,u_2\right).
\end{equation}
\end{lemma}

By following the proof of Theorem $(3.2)$ in \cite{Beck}, and replacing $i$ by $\lambda$,\ $j$ by $\mu$\ we obtain the next lemma.

\begin{lemma}{(QFT Plancherel)}\\
Let $f \in L^2({\mathbb R}^2,{\mathbb H})$, then

\begin{equation}
\int_{{\mathbb R}^2}{{\left\|{\mathcal F}^{\lambda,\mu }\left\{f\right\}\left(u\right)\right\|}^2_Q}du=4{\pi }^2\int_{{\mathbb R}^2}{{\left|f(t)\right|}^2_Q}dt.
\end{equation}

\end{lemma}

\begin{lemma}\leavevmode
 If   $f\in L^2({\mathbb R}^2,{\mathbb H}), \frac{{\partial }^m{\partial }^n}{{\partial t}^m_1{\partial t}^n_2}$f exist and are in   $L^2\left({\mathbb R}^2,{\mathbb H}\right)$ for \ $m,n\in {\mathbb N}_0$\
 then
 \begin{equation}
{\mathcal F}^{\lambda,\mu }\left\{\frac{{\partial }^{m+n}}{{\partial t}^m_1{\partial t}^n_2}f\right\}\left(u\right)={(\lambda u_1)}^m\ {\mathcal F}^{\lambda,\mu }\left\{f\right\}\left(u\right)  {(\mu u_2)}^n.
\end{equation}
\end{lemma}
Proof. See (\cite{Bulow}, Thm. 2.10).
\begin{lemma}\label{inverse_QFT}[Inverse QFT] (see\cite{Hitzer3})

If  $f\in L^1\left({\mathbb{R}}^2, {\mathbb{H}}\right), and \ {\mathcal F}^{\lambda,\mu}\{f\}\in L^1\left({\mathbb{R}}^2,{\mathbb{H}}\right)$, then the two-sided QFT is an invertible transform and its inverse is given
by\\
\begin{equation}
f(t)= \frac{1}{{(2\pi )}^2} \int_{{\mathbb{R}}^2}{e^{\lambda u_1 t_1}} {\mathcal F}^{\lambda,\mu} \{f(t)\}(u) e^{\mu u_2 t_2}du.
\end{equation}

\end{lemma}
\section{The offset quaternionic linear canonical transform }

Morais et al \cite{Morais} introduce the quaternionic linear canonical transform (QLCT). They consider two real  matrixes
\[A_1=\left[ \begin{array}{cc}
a_1 & b_1 \\
c_1 & d_1 \end{array}
\right], A_1=\left[ \begin{array}{cc}
a_2 & b_2 \\
c_2 & d_2 \end{array}
\right]\in  {\mathbb R}^{2\times 2}.\]

 with \ \ \ $a_1d_1-b_1c_1=1,\ a_2d_2-b_2c_2=1,$\\
Eckhard Hitzer \cite{Hitzer2} generalize the definitions of \cite{Morais} to be:
the two-sided QLCT of signals f $\in L^1({\mathbb R}^2,{\mathbb H})$, is defined by

\begin{equation}\label{QLCT}
{\mathcal L}^{\lambda,\mu }_{A_1,A_2}\{f\}(u)=
\left\{
  \begin{array}{ll}
    \int_{{\mathbb R}^2}{K^{\lambda }_{A_1}\left(t_1,u_1\right)}f\left(t\right)K^{\mu }_{A_2}\left(t_2,u_2\right)dt,\quad  b_1,b_2\ne 0; \\
    \,   \\
    \sqrt{d_1} e^{\lambda \frac{c_1d_1}{2}{u_1}^2}f(d_1 u_1,t_2)K^{\mu }_{A_2}\left(t_2,u_2\right), \quad b_1=0,b_2\ne 0; \\
    \, \\
    \sqrt{d_2}{K^{\lambda }_{A_1}\left(t_1,u_1\right)}f(t_1,d_2u_2)e^{\mu \frac{c_2d_2}{2}u_2^2},  \quad b_1\ne 0, b_2=0; \\
    \, \\
    \sqrt{d_1d_2}e^{\lambda \frac{c_1d_1}{2}{u_1}^2}f(d_1u_1,d_2u_2)e^{\mu \frac{c_2d_2}{2}u_2^2}, \quad b_1=b_2=0.\\
  \end{array}
\right.
\end{equation}

with $\lambda,\mu \in {\mathbb H},$ denote two pure unit quaternions, ${\lambda }^2={\mu }^2=-1$,\ including the cases $\lambda =\pm \mu,$
$$K^{\lambda }_{A_1}\left(t_1,u_1\right)=\frac{1}{\sqrt{\lambda 2\pi b_1}} e^{{\lambda (a_1t^2_1-2t_1u_1+d_1u^2_1)}/{2b_1}}, \qquad
 K^{\mu }_{A_2}\left(t_2,u_2\right)=\frac{1}{\sqrt{\mu 2\pi b_2}} e^{{\mu (a_2t^2_2-2t_2u_2+d_2u^2_2)}/{2b_2}},$$
In \cite{Morais},  the properties of the right-sided QLCT and its uncertainty principles  are studied in  detail.
El Haoui et al \cite{Haoui3} introduced and studied the QOLCT, and established its properties and uncertainty principles.
Let's give  the definitions of Quaternionic offset linear canonical transform as follows:
\begin{definition}
Let  $A_l=\left[\left| \begin{array}{cc}
a_l & b_l \\
c_l & d_l \end{array}
\right| \begin{array}{c}
{\tau }_l \\
{\eta }_l \end{array}
\right]$,

the parameters   $a_l, b_l, c_l,d_l,\ {\tau }_l,\ {\eta }_l\in {\mathbb R}$ such that $a_ld_l-b_lc_l=1$, for $l=1,2$

the two-sided quaternionic offset linear canonical transform (QOLCT) of a signal $f \in L^1({\mathbb R}^2, {\mathbb H})$, is given by

$${\mathcal O}^{\lambda,\mu }_{A_1,A_2}\{f(t)\}(u)=
\left\{
  \begin{array}{ll}
    \int_{{\mathbb R}^2}{K^{\lambda }_{A_1}\left(t_1,u_1\right)}f\left(t\right)K^{\mu }_{A_2}\left(t_2,u_2\right)dt,\quad  b_1,b_2\ne 0; \\
    \,   \\
    \sqrt{d_1}{e}^{\lambda (\frac{c_1d_1}{2}{\left(u_1-\tau_1\right)}^2+{\ u}_1\tau_1)}f\left(d_1{(u}_1-\tau_1\right),t_2)K^{\mu }_{A_2}\left(t_2,u_2\right), \quad b_1=0,b_2\ne 0; \\
    \, \\
    \sqrt{d_2}{K^{\lambda }_{A_1}\left(t_1,u_1\right)}f({t_1,d}_2{(u}_2-\tau_2))e^{\mu (\frac{c_2d_2}{2}{(u_2-\tau_2)}^2+{\ u}_2{\tau }_{2)}},  \quad b_1\ne 0, b_2=0; \\
    \, \\
    \sqrt{d_1d_2}e^{\lambda(\frac{c_1d_1}{2}{\left(u_1-\tau_1\right)}^2+{u}_1\tau_1)}f(d_1({u}_1-{\tau}_1),d_2(u_2-\tau_2))e^{\mu(\frac{c_2d_2}2{(u_2-\tau_2)}^2+{u}_2{\tau}_2)}, \\\quad b_1=b_2=0.
  \end{array}
\right.$$

Where
\begin{equation}\label{kernel1}
  K^{\lambda }_{A_1}\left(t_1,u_1\right)=\frac{1}{\sqrt{\lambda 2\pi b_1}}\ e^{\lambda (a_1t^2_1-2t_1(u_1-\tau_1)-2{u_1(d}_1\tau_1-b_1\eta_1)+d_1{(u}^2_1+{\tau }^2_1))\frac{1}{2b_1}}, \quad for \quad b_1\ne 0,
  \end{equation}
and\\
\begin{equation}\label{kernel2}
K^{\mu }_{A_2}\left(t_2,u_2\right)=\ \frac{1}{\sqrt{\mu 2\pi b_2}}\  e^{\mu (a_2t^2_2-2t_2(u_2-\tau_2)-2{u_2(d}_2\tau_2-b_2\eta_2)+d_2{(u}^2_2+{\tau }^2_2))\frac{1}{2b_2}}, \quad for\quad b_2\ne 0\ \,
  \end{equation}
with $$\frac{1}{\sqrt{\lambda }}=e^{-\lambda \frac{\pi }{4}},\quad\frac{1}{\sqrt{\mu }} =e^{-\mu \frac{\pi }{4}}.$$

The left-sided and right-sided QOLCTs can be defined  by placing the two kernel factors both on the left or on the right, respectively.
\end{definition}

We remark that,  when $\tau_1=\tau_2=\ \eta_1=\eta_2=$0, the two-sided QOLCT reduces to the QLCT.

Also, when    $A_1=A_2=\left[\left| \begin{array}{cc}
0 & 1 \\
-1 & 0 \end{array}
\right| \begin{array}{c}
0 \\
0 \end{array}
\right]$, the conventional two-sided QFT is recovered. Namely,

\begin{eqnarray*}
  {\mathcal O}^{\lambda,\mu }_{A_1,A_2}\{f(t)\}(u) &=& \frac{1}{\sqrt{\lambda 2\pi }}(\int_{{\mathbb R}^2}{ e^{-\lambda t_1u_1}}f\left(t\right)e^{-\mu t_2u_2}dt)\frac{1}{\sqrt{\mu 2\pi }} \\
   &=& \frac{1}{2\pi }e^{-\lambda \frac{\pi }{4}}{\mathcal F}^{\lambda,\mu }\left\{f\right\}\left(u_1,u_2\right)\ e^{-\mu \frac{\pi }{4}},
\end{eqnarray*}

where ${\ \mathcal F}^{\lambda,\mu }\left\{f\right\}$ is the QFT of  $f$ given by \eqref{QFT}.

The following lemma gives the relationships of two-sided QOLCTs and two-sided QFTs of 2D quaternion-valued signals.

\begin{lemma}
The QOLCT of a signal $f \in L^1({\mathbb R}^2, {\mathbb H})$\ can be reduced to the QFT
\begin{equation}
  \label{h-function}
{\mathcal O}^{\lambda,\mu }_{A_1,A_2}\left\{f\left(t\right)\right\}\left(u_1,u_2\right)= {\mathcal F}^{\lambda,\mu}\left\{h(t)\right\}\left(\frac{u_1}{b_1},\frac{u_2}{b_2}\right)\ ,
   \end{equation}
with


$$h(t) = \frac{1}{\sqrt{2\pi {\lambda b}_1}}e^{\lambda[-\frac{1}{b_1}{u_1(d}_1\tau_1-b_1\eta_1)+\frac{d_1}{2b_1}{(u}^2_1+{\tau }^2_1)+\frac{1}{b_1}t_1\tau_1+\frac{a_1}{2b_1}t^2_1]}f(t)$$

$$   \times e^{\mu[-\frac{1}{b_2}{u_2(d}_2\tau_2-b_2\eta_2)+\frac{d_2}{2b_2}{(u}^2_2+{\tau }^2_2)+ \frac{1}{b_2}t_2{\tau}_2+\frac{a_2}{2b_2}t^2_2]}\frac{1}{\sqrt{2\pi \mu b_2}}.$$

\end{lemma}

By using lemma \ref{inverse_QFT} and \eqref{h-function}, we get the inversion formula for the QOLCT,
\begin{theorem}\label{QOLCT-inv}
If $f$  \  and ${\mathcal O}^{\lambda,\mu}_{A_1,A_2}\left\{f\right\}$ \ are in $ \ L^1({\mathbb R}^2,{\mathbb H}),$
then the inverse transform of the QOLCT can be derived from that of the QFT, and we have

\begin{eqnarray*}
f(t)&=& \int_{{\mathbb R}^2}{\overline{K^{\lambda }_{A_1}\left(t_1,u_1\right)}{\mathcal O}^{\lambda,\mu}_{A_1,A_2}\left\{f(t)\right\}\left(u_1,u_2\right)\overline{K^{\mu }_{A_2}\left(t_2,u_2\right)}}du.
\end{eqnarray*}

\end{theorem}
\begin{theorem}{(Plancherel's theorem of the QOLCT)}\label{QOLCT-Placherel}\\
Every 2D quaternion-valued signal $f\in L^2({\mathbb R}^2,{\mathbb H})$ and its QOLCT are related to the Plancherel identity in the following way:
\begin{equation}\label{QOLCT-planch}
 {\left\|{\mathcal O}^{\lambda,\mu }_{A_1,A_2}\left\{f\right\}\right\|}_{2,Q}={\left|f\right|}_{2,Q}.
\end{equation}
\end{theorem}
\section{Wigner-Ville distribution associated with quaternionic offset linear canonical transform}

The Fourier transform is a powerful tool to study the stationary signals, but it has become not sufficient for characterize the non-stationary signals. However, in practice, most natural signals are non stationary. In order to study a non stationary signal the Wigner-Ville distribution has become a suite tool for the analysis of the non stationary signals.\\
In this section, we are going to give the definition of Wigner-Ville distribution associated with the quaternionic offset  linear canonical transform WVD-QOLCT, then, we will investigate its important properties, and establish
the Heisenberg uncertainty principle, Poisson summation formula and Lieb's theorem related for the WVD-QOLCT.
 \begin{definition}\leavevmode\par
Let  $A_l=\left[\left| \begin{array}{cc}
a_l & b_l \\
c_l & d_l \end{array}
\right| \begin{array}{c}
{\tau }_l \\
{\eta }_l \end{array}
\right]$,
with  $a_l, b_l, c_l,d_l,\ {\tau }_l,\ {\eta }_l\in {\mathbb R}$ such that $a_ld_l-b_lc_l=1$, for $l=1,2$.

The Wigner-Ville distribution associated with the two-sided quaternionic offset linear canonical transform (WVD-QOLCT) of a signal $f \in L^1({\mathbb R}^2, {\mathbb H})$, is given by



$${\mathcal W}_{f,g }^{A_1,A_2}(t,u)=
\left\{
  \begin{array}{ll}
   \int_{{\mathbb R}^2}{K^{\lambda }_{A_1}(s_1,u_1)}f(t+\frac{s}{2})\overline{g}(t-\frac{s}{2})K^{\mu }_{A_2}(s_2,u_2)ds,\qquad b_1,b_2\ne 0, \,   \\ \\
    \sqrt{d_1}{e}^{\lambda (\frac{c_1d_1}{2}{\left(u_1-\tau_1\right)}^2+{\ u}_1\tau_1)}f(t_1+\frac{d_1{(u}_1-\tau_1)}{2},t_2+\frac{s_2}{2})\\
	\times\overline{g}(t_1-\frac{d_1{(u}_1-\tau_1)}{2},t_2-\frac{s_2}{2})K^{\mu }_{A_2}\left(s_2,u_2\right), \quad b_1=0,b_2\ne 0; \\
    \, \\
    \sqrt{d_2}{K^{\lambda }_{A_1}\left(s_1,u_1\right)}f(t_1+\frac{s_1}{2},t_2+\frac{d_2{(u}_2-\tau_2)}{2})\\ \ \ \times\overline{g}(t_1-\frac{s_1}{2},t_2-\frac{d_2{(u}_2-\tau_2)}{2})  e^{\mu (\frac{c_2d_2}{2}{(u_2-\tau_2)}^2+{\ u}_2{\tau }_{2)}}  ,  \quad b_1\ne 0, b_2=0; \\
    \, \\
    \sqrt{d_1d_2}e^{\lambda(\frac{c_1d_1}{2}{\left(u_1-\tau_1\right)}^2+{u}_1\tau_1)}f(t_1+\frac{d_1({u}_1-{\tau}_1}{2}),t_2+\frac{d_2(u_2-\tau_2}{2})) \\
	\ \ \times \overline{g} (t_1-\frac{d_1({u}_1-{\tau}_1}{2}),t_2-\frac{d_2(u_2-\tau_2}{2}))    e^{\mu(\frac{c_2d_2}2{(u_2-\tau_2)}^2+{u}_2{\tau}_2)},  \quad b_1=b_2=0.
  \end{array}
\right.$$

where
  $K^{\lambda }_{A_1}(s_1,u_1),$\ and $K^{\mu }_{A_2}\left(s_2,u_2\right),$ \ are given respectively  by \eqref{kernel1}, and \eqref{kernel2}.

\end{definition}

\begin{remark}\label{rem}
  It's clear that if we take $h_{f,g}(t,s)=f(t+\frac{s}{2})\overline{g}(t-\frac{s}{2})$ for all $t,s\in\mathbb{R}^2,$\\ we have,

  \begin{equation}\label{wigner-offsset}
\mathcal{W}_{f,g }^{A_1,A_2}(t,u)=\mathcal{O}^{\lambda,\mu}_{A_1,A_2}\{h_{f,g}(t,s)\}(u).
\end{equation}

 We note that when we take $\tau_l=\eta_l=0$, $l=1,2$ the WVD-QOLCT reduces to the WVD-QLCT\cite{Bahri2}.\\
 \end{remark}

 And by using \eqref{h-function}, we obtain the relation between WVD-QOLCT and QFT:

\begin{lemma}
\begin{equation}\label{hwv-function}
{\mathcal W}_{f,g }^{A_1,A_2}(t,u)= {\mathcal F}^{\lambda,\mu}\left\{k_{f,g}(t,s)\right\}\left(\frac{u_1}{b_1},\frac{u_2}{b_2}\right),
\end{equation}
\end{lemma}
where
  $k_{f,g}(t,s) = \frac{1}{\sqrt{2\pi {\lambda b}_1}}e^{\lambda[-\frac{1}{b_1}{u_1(d}_1\tau_1-b_1\eta_1)+\frac{d_1}{2b_1}{(u}^2_1+{\tau }^2_1)+\frac{1}{b_1}s_1\tau_1+\frac{a_1}{2b_1}s^2_1]}h_{f,g}(t,s) \\
   \ \ \ \ \times e^{\mu[-\frac{1}{b_2}{u_2(d}_2\tau_2-b_2\eta_2)+\frac{d_2}{2b_2}{(u}^2_2+{\tau }^2_2)+ \frac{1}{b_2}s_2{\tau}_2+\frac{a_2}{2b_2}s^2_2]}\frac{1}{\sqrt{2\pi \mu b_2}}.$

 Now, we give the inversion formula for the WVD-QOCLT

\begin{theorem}\leavevmode\par
  If $f,g$ and $ \mathcal{W}_{f,g }^{A_1,A_2}$ are in $L^2(\mathbb{R}^2,\mathbb{H})$, then,  the inverse transform of QWVD-OCLT is given by

\begin{equation}\label{inversion}
 f(v)= \frac{1}{|g|^2_{2,Q}}\int_{\mathbb{R}^2}\int_{{\mathbb R}^2}{\overline{K^{\lambda }_{A_1}\left(\frac{v_1+\varepsilon_1}{2},u_1\right)}{\mathcal W}_{f,g}^{A_1,A_2}(\frac{v+\varepsilon}{2},u)\overline{K^{\mu }_{A_2}\left(\frac{v_2+\varepsilon_2}{2},u_2\right)}}g(\varepsilon)dud\varepsilon
\end{equation}

\end{theorem}

\begin{proof}
  By the equation \eqref{wigner-offsset}
$$\mathcal{W}_{f,g }^{A_1,A_2}(t,u)=\mathcal{O}^{\lambda,\mu}_{A_1,A_2}\{f(t+\frac{.}{2})\overline{g}(t-\frac{.}{2})\}(u).$$
then, by theorem\ref{QOLCT-inv}, we obtain

\begin{eqnarray*}
  f(t+\frac{s}{2})\overline{g}(t-\frac{s}{2})&=&  \int_{{\mathbb R}^2}{\overline{K^{\lambda }_{A_1}\left(t_1,u_1\right)}{\mathcal W}_{f,g}^{A_1,A_2}(t,u)\overline{K^{\mu }_{A_2}\left(t_2,u_2\right)}}du.
\end{eqnarray*}
By taking $v=t+\frac{s}{2}$ and $\varepsilon=t-\frac{s}{2}$ we get $t=\frac{v+\varepsilon}{2}$
and
\begin{equation}\label{inver1}
f(v)\overline{g}(\varepsilon)= \int_{{\mathbb R}^2}{\overline{K^{\lambda }_{A_1}\left(\frac{v_1+\varepsilon_1}{2},u_1\right)}{\mathcal W}_{f,g}^{A_1,A_2}(\frac{v+\varepsilon}{2},u)\overline{K^{\mu }_{A_2}\left(\frac{v_2+\varepsilon_2}{2},u_2\right)}}du.
\end{equation}
Multiplying both sides of \eqref{inver1} from the right by $g$ and integrating with respect to $d\varepsilon$ we get

\begin{equation}\label{inver2}
f(v)\int_{\mathbb{R}^2}|g(\varepsilon)|^2 d\varepsilon= \int_{\mathbb{R}^2}\int_{{\mathbb R}^2}{\overline{K^{\lambda }_{A_1}\left(\frac{v_1+\varepsilon_1}{2},u_1\right)}{\mathcal W}_{f,g}^{A_1,A_2}(\frac{v+\varepsilon}{2},u)\overline{K^{\mu }_{A_2}\left(\frac{v_2+\varepsilon_2}{2},u_2\right)}}g(\varepsilon)dud\varepsilon.
\end{equation}
Consequently,
\begin{equation}\label{inver3}
 f(v)= \frac{1}{|g|^2_{2,Q}}\int_{\mathbb{R}^2}\int_{{\mathbb R}^2}{\overline{K^{\lambda }_{A_1}\left(\frac{v_1+\varepsilon_1}{2},u_1\right)}{\mathcal W}_{f,g}^{A_1,A_2}(\frac{v+\varepsilon}{2},u)\overline{K^{\mu }_{A_2}\left(\frac{v_2+\varepsilon_2}{2},u_2\right)}}g(\varepsilon)dud\varepsilon.
\end{equation}

\end{proof}

The following theorem gives  the Plancherel's identity fo for the WVD-QOLCT,
\begin{theorem}[Plancherel's theorem for WVD-QOLCT]\leavevmode\par
  Let $f,g \in L^2(\mathbb{R}^2,\mathbb{H})$, then we have,
\begin{equation}\label{WVD-Plancherel}
  {\left\|\mathcal{W}_{f,g }^{A_1,A_2}\right\|}_{2,Q}= \left|f\right|^2_{2,Q}   \left|g\right|^2_{2,Q}.
\end{equation}
\end{theorem}

\begin{proof}
  We have  by the equality \eqref{wigner-offsset}
$$\mathcal{W}_{f,g }^{A_1,A_2}(t,u)=\mathcal{O}^{\lambda,\mu}_{A_1,A_2}\{h_{f,g}(t,.)\}(u),$$
 and the Plancherel formula for the QOLCT \eqref{QOLCT-planch}
$${\left\|{\mathcal O}^{\lambda,\mu }_{A_1,A_2}\left\{f\right\}\right\|}_{2,Q}={\left|f\right|}_{2,Q}.$$

So
\begin{eqnarray*}
  {\left\|\mathcal{W}_{f,g }^{A_1,A_2}\right\|}_{2,Q} &=&{\left\|\mathcal{O}^{\lambda,\mu}_{A_1,A_2}\{h_{f,g}\}\right\|}_{2,Q} \\
   &=& {\left|h_{f,g}\right|}_{2,Q} \\
   &=& \left(\int_{{\mathbb R}^2}\int_{{\mathbb R}^2}{{\left| f(t+\frac{s}{2})\overline{g}(t-\frac{s}{2})\right|}^2_Q}dt ds\right)^{\frac{1}{2}} \\
   &=& \left(\int_{{\mathbb R}^2}\int_{{\mathbb R}^2}{\left| f(u)\overline{g}(v)\right|}^2_Q dudv\right)^{\frac{1}{2}} \\
   &=&\int_{{\mathbb R}^2}\left| f(u)\right|^2du \int_{{\mathbb R}^2}\left| g(v)\right|^2 dv\\
   &=&\left|f\right|^2_{2,Q}   \left|g\right|^2_{2,Q}.
\end{eqnarray*}

\end{proof}

\begin{definition}
  A couple $\alpha=(\alpha_1,\alpha_2)$ of non negative integers is called a multi-index. One denotes
  $$|\alpha|=\alpha_1+\alpha_2 ~~ and ~~ \alpha!=\alpha_1!\alpha_2!$$
  and, for $x\in \mathbb{R}^2$
  $$x^{\alpha}=x_1^{\alpha_1}x_2^{\alpha_2}$$
  Derivatives are conveniently expressed by multi-indices
 $$ \partial^{\alpha}=\frac{\partial^{|\alpha|}}{\partial x_1^{\alpha_1}\partial x_2^{\alpha_2}}$$

\end{definition}
Next, we obtain the Schwartz space as (\cite{Kou})
$$\mathcal{S}(\mathbb{R}^2,\mathbb{H})=\{f\in C^{\infty}(\mathbb{R}^2,\mathbb{H}): sup_{x\in\mathbb{R}^2}(1+|x|^k)|\partial^{\alpha}f(x)|<\infty\},$$
where  $C^{\infty}(\mathbb{R}^2,\mathbb{H})$ is the set of smooth function from $\mathbb{R}^2$ to $\mathbb{H}$.

The following theorem is the Heisenberg's theorem for QOLCT (see \cite{Haoui3}),
\begin{theorem}[Heisenberg QOLCT]\label{heisenberg-QOLCT}\leavevmode\par

Suppose that  $f, \ \frac{\partial }{\partial s_k} f,\ s_k f\in L^2({\mathbb R}^2,{\mathbb H})\ $\ for  $k=1,2,$ \\
   then

\begin{equation}\label{heis-QOLCT1}
{\left|s_kf\left(s\right)\right|}^2_{2,Q}  {\left\|\frac{{\xi }_k}{2\pi b_k}\mathcal{O}^{\lambda,\mu}_{A_1,A_2}\left\{f\left(s\right)\right\}\left(\xi \right)\right\|}^2_{2,Q}\ge \frac{1}{{16\pi }^2}{\left|f\left(s\right)\right|}^4_{2,Q}.
\end{equation}
\end{theorem}

The next theorem states the Heisenberg's uncertainty principle  for the WVD-QOLCT.
\begin{theorem}\leavevmode\par
Let $f,g \in \mathcal{S}(\mathbb{R}^2,\mathbb{H})$.  We have the following inequality
\begin{equation}\label{heisenberg}
\left(\int_{\mathbb{R}^2}\int_{\mathbb{R}^2}|s_k f(t+\frac{s}{2})\overline{g}(t-\frac{s}{2})|^2_{Q}dsdt \right) \left(\int_{\mathbb{R}^2}\int_{\mathbb{R}^2}\|\frac{\xi_k}{2\pi b_k}\mathcal{W}_{f,g}^{A_1,A_2}(t,\xi)\|^2_{Q}d\xi dt\right)
\geq\frac{1}{16\pi^2}|f|_{2,Q}^4|g|_{2,Q}^{4}.
\end{equation}

\end{theorem}

\begin{proof}
Let $h_{f,g}$,\ be rewritten as in remark \ref{rem}.\\
As $f,g \in \mathcal{S}(\mathbb{R}^2,\mathbb{H})$, we obtain  that $h_{f,g}(.,s)\in L^2(\mathbb{R}^2,\mathbb{H}).$\\
Therefore by applying \eqref{heisenberg}, we get

$${\left|s_k h_{f,g}\left(.,s\right)\right|}^2_{2,Q}  {\left\|\frac{{\xi }_k}{2\pi b_k}\mathcal{O}^{\lambda,\mu}_{A_1,A_2}\left\{h_{f,g}(.,s)\right\}\left(\xi \right)\right\|}^2_{2,Q}\ge \frac{1}{{16\pi }^2}{\left|h_{f,g}(.,s)\right|}^4_{2,Q}.$$
According to  \eqref{wigner-offsset}, we obtain

$${\left|s_k h_{f,g}\left(t,s\right)\right|}^2_{2,Q}  {\left\|\frac{{\xi }_k}{2\pi b_k}\mathcal{W}^{A_1,A_2}_{f,g}(t,\xi)\right\|}^2_{2,Q}\ge \frac{1}{{16\pi }^2}{\left|h_{f,g}(t,s)\right|}^4_{2,Q}.$$

Then, we have,
\begin{equation}\label{heisen3}
|s_k f(t+\frac{s}{2})\overline{g}(t-\frac{s}{2})|^2_{2,Q}\|\frac{\xi_k}{2\pi b_k}\mathcal{W}_{f,g}^{A_1,A_2}(t,\xi)\|^2_{2,Q}\ge \frac{1}{{16\pi }^2}{\left|h_{f,g}(t,s)\right|}^4_{2,Q}.
\end{equation}
By taking the square root on both sides of \eqref{heisen3}and integrating both sides with respect to $dt,$ we get

$$\int_{\mathbb{R}^2}\left(\left(\int_{\mathbb{R}^2}|s_k f(t+\frac{s}{2})\overline{g}(t-\frac{s}{2})|^2_{Q}ds \right)^{\frac{1}{2}} \left(\int_{\mathbb{R}^2}\|\frac{\xi_k}{2\pi b_k}\mathcal{W}_{f,g}^{A_1,A_2}(t,\xi)\|^2_{Q}d\xi\right)^{\frac{1}{2}}\right) dt$$
\begin{equation}\label{heisen4}
\geq\frac{1}{4\pi}\int_{\mathbb{R}^2}\int_{\mathbb{R}^2}|h_{f,g}(t,s)|^2_{Q} dsdt.
\end{equation}

Now, by applying the Schwartz's inequality to the left hand side of \eqref{heisen4}, and using \eqref{WVD-Plancherel},
 we obtain

$$\left(\int_{\mathbb{R}^2}\int_{\mathbb{R}^2}|s_k f(t+\frac{s}{2})\overline{g}(t-\frac{s}{2})|^2_{Q}dsdt \right)^{\frac{1}{2}} \left(\int_{\mathbb{R}^2}\int_{\mathbb{R}^2}\|\frac{\xi_k}{2\pi b_k}\mathcal{W}_{f,g}^{A_1,A_2}(t,\xi)\|^2_{Q}d\xi dt\right)^{\frac{1}{2}}
\geq\frac{1}{4\pi}|f|_{2,Q}^2|g|_{2,Q}^{2}.$$
Therefore, the proof is complete.

\end{proof}

\subsection{Poisson summation formula}\leavevmode\par

It is well known that, the Poisson summation formula play an important role in mathematics, due to its various applications in signal processing. In this section we generalize the above mentioned formula into WVD-QOLCT domaine.

\begin{proposition}(see\  \cite{Beck} ) \label{poisson-QFT}
  Let $f\in L^1(\mathbb{R}^2,\mathbb{H})$, then
\begin{equation}\label{poisson-formula}
  \sum_{(k_1,k_2)\in\mathbb{Z}^2}f(s_1+k_2,s_2+k_2)=\sum_{(k_1,k_2)\in\mathbb{Z}^2}e^{2\pi ik_1s_1}\widehat{f}(k_1,k_2)e^{2\pi jk_2s_2}
\end{equation}
where $\widehat{f}$ is the QFT of $f$ defined by $\widehat{f}(\xi)=\int_{\mathbb{R}^2}e^{-2\pi is_1\xi_1}f(s)e^{-2\pi jt_2\xi_2}ds.$
\end{proposition}
Now, we give a version of Poisson summation formula for the WVD-QOLCT,
\begin{theorem}
  Let $f,g\in L^2(\mathbb{R}^2,\mathbb{H})$, then
$$ \sum_{(k_1,k_2)\in\mathbb{Z}^2}e^{\frac{i}{b_1}(s_1+k_1)\tau_1+i\frac{a_1}{2b_1}(s_1+k_1)^2}f(t+\frac{s}{2})\overline{g}(t-\frac{s}{2})e^{\frac{j}{b_2}(s_2+k_2)+j\frac{a_2}{2b_2}(s_2+k_2)^2}=$$
$$\sqrt{2\pi ib_1}[\sum_{(k_1,k_2)\in\mathbb{Z}^2 }e^{2\pi i k_1s_1}e^{2\pi i k_1(d_1\tau_1-b_1\eta_1)-i\frac{d_1}{2b_1}(4\pi^2 b_1^2 k_1^2+\tau_1^2)}\mathcal{W}_{f,g}^{A_1,A_2}(t,(2\pi b_1 k_1,2\pi b_2 k_2))$$ $$\times e^{2\pi j k_2 s_2} e^{2\pi j k_2(d_2\tau_2-b_2\eta_2)-j\frac{d_2}{2b_2}(4\pi^2 b_2^2k_2^2+\tau_2^2)}]\sqrt{2\pi j b_2}$$
\end{theorem}

\begin{proof}
  Let $\omega_{f,g}(t,s)=e^{i\frac{1}{b_1}s_1\tau_1+i\frac{a_1}{2b_1}s^2_1}f(t+\frac{s}{2})\overline{g}(t-\frac{s}{2})e^{j\frac{1}{b_2}s_2\tau_1+j\frac{a_2}{2b_2}s^2_2}.$\\
As $f,g\in L^2(\mathbb{R}^2,\mathbb{H}) ,$ we have by H\"{o}lder's inequality $\omega_{f,g}\in L^{1}(\mathbb{R}^2,\mathbb{H})$, then by proposition \ref{poisson-QFT} we have
$$\sum_{(k_1,k_2)\in\mathbb{Z}^2}\omega_{f,g}(t,s_1+k_2,s_2+k_2)=\sum_{(k_1,k_2)\in\mathbb{Z}^2}e^{2\pi i k_1 s_1}\mathcal{F}^{i,j}\{\omega_{f,g}(t,(s_1,s_2))\}(2\pi k_1,2\pi k_2)e^{2\pi jk_2s_2}.$$

Applying  \eqref{hwv-function} leads to

$$\sum_{(k_1,k_2)\in\mathbb{Z}^2 } e^{\frac{i}{b_1}(s_1+k_1)\tau_1+i\frac{a_1}{2b1}(s_1+k_1)^2}f(t+\frac{s}{2})\overline{g}(t-\frac{s}{2})e^{\frac{j}{b_2}(s_2+k_2)\tau_2+j\frac{a_2}{2b_2}(s_2+k_2)^2}=$$
$$\sqrt{2\pi ib_1}\sum_{(k_1,k_2)\in\mathbb{Z}^2 }e^{2\pi i k_1s_1}e^{2\pi i k_1(d_1\tau_1-b_1\eta_1)-i\frac{d_1}{2b_1}(4\pi^2 b_1^2 k_1^2+\tau_1^2)}\mathcal{W}_{f,g}^{A_1,A_2}(t,(2\pi b_1 k_1,2\pi b_2 k_2))$$ $$\times e^{2\pi j k_2 s_2} e^{2\pi j k_2(d_2\tau_2-b_2\eta_2)-j\frac{d_2}{2b_2}(4\pi^2 b_2^2k_2^2+\tau_2^2)}\sqrt{2\pi j b_2}.$$

\end{proof}

\subsection{Lieb's theorem}\leavevmode\par
In this part of this paper, we are going to give a version of Lieb's theorem for the WVD-QOLCT. \\
In the following theorem\cite{Bahri3}, we state Lieb's theorem related to the QLCT.

\begin{theorem}
  If $1\leq p\leq2$ and let $q$ be such that $\frac{1}{p}+\frac{1}{q}=1$, then, for all $f\in L^{p}(\mathbb{R}^2,\mathbb{H})$, it holds that
\begin{equation}\label{Lieb-LCT}
|{\mathcal L}^{i,j}_{A_1,A_2}\{f\}|_{q,Q}\leq \frac{|b_1 b_2|^{\frac{-1}{2}+\frac{1}{q}}}{2\pi}|f|_{p,Q}
\end{equation}
\end{theorem}

\begin{proof}
  For the proof see \cite{Bahri3}.
\end{proof}

\begin{theorem}[Lieb's theorem associated with the WVD-QOLCT]\leavevmode\par\label{Lieb-WVD}
 Let $2\leq p<\infty$ and $f,g\in L^{2}(\mathbb{R}^2,\mathbb{H})$. Then

  $$\int_{\mathbb{R}^2}\int_{\mathbb{R}^2}|W_{f,g}^{A_1,A_2}(t,u)|_{Q}^{p}dudt\leq C \frac{|b_1 b_2|^{\frac{-p}{2}+1}}{(2\pi)^{p}} |f|^p_{2,Q}|g|^p_{2,Q},$$

where $C$ is a positive constant.
\end{theorem}
Before proving this theorem, we need the following lemma,
\begin{lemma}
Let  $$A_l=\left[\left| \begin{array}{cc}
a_l & b_l \\
c_l & d_l \end{array}
\right| \begin{array}{c}
{\tau }_l \\
{\eta }_l \end{array}
\right],\qquad
and\qquad
B_l=\left(
  \begin{array}{cc}
   a_l & b_l \\
   c_l & d_l \\
  \end{array}
\right),
$$ with $a_l d_l-b_l c_l=1$ for $l=1,2$ .\\
 For $f\in L^{1}(\mathbb{R}^2,\mathbb{H})$, we have the relation:
\begin{equation}\label{relation-LCT-QOLCT}
  \mathcal{O}^{\lambda,\mu}_{A_1,A_2}\{f\}(u)=e^{\lambda(2t_1\tau_1-2u_1(d_1\tau_1-b_1\eta_1))}e^{\mu d_1\frac{\tau^2_1}{2b_1}}{\mathcal L}_{B_1,B_2}\{f\}(u)e^{\mu d_2\frac{\tau^2_2}{2b_2}}e^{\mu(2t_2\tau_2-2u_2(d_2\tau_2-b_2\eta_2))}.
\end{equation}

\end{lemma}

\begin{proof}
  To prove this lemma we just use the definitions of the QOLCT and QLCT to obtain the result.
\end{proof}

Now we give a demonstration of the theorem \ref{Lieb-WVD}
\begin{proof}\leavevmode\par
  We have by the equation  \eqref{relation-LCT-QOLCT},
\begin{eqnarray*}
  \left(\int_{\mathbb{R}^2}|{\mathcal W}_{f,g}^{A_1,A_2}(t,u)|_{Q}^{p}du\right)^{\frac{1}{p}} &=& \left(\int_{\mathbb{R}^2}|{\mathcal O}^{\lambda ,\mu}_{A_1,A_2}\{f(t+\frac{s}{2})\overline{g}(t-\frac{s}{2})\}(u)|^{p}_{Q}du\right)^{\frac{1}{p}} \\
   &=& \left(\int_{\mathbb{R}^2}|{\mathcal L}^{\lambda ,\mu}_{B_1,B_2}\{f(t+\frac{s}{2})\overline{g}(t-\frac{s}{2})\}(u)|_{Q}^{p}du\right)^{\frac{1}{p}}\\
   &\leq& \frac{|b_1 b_2|^{\frac{-1}{2}+\frac{1}{p}}}{2\pi} \left(\int_{\mathbb{R}^2}|f(t+\frac{s}{2})\overline{g}(t-\frac{s}{2})|^q_{Q}ds\right)^{\frac{1}{q}}.
\end{eqnarray*}
In the last equality we used \eqref{Lieb-LCT}. \\
Furthermore,
$$\int_{\mathbb{R}^2}|{\mathcal W}_{f,g}^{A_1,A_2}(t,u)|_{Q}^{p}du\leq \frac{|b_1 b_2|^{\frac{-p}{2}+1}}{(2\pi)^{p}}\left(\int_{\mathbb{R}^2}|f(t+\frac{s}{2})\overline{g}(t-\frac{s}{2})|_{Q}^{q}ds\right)^{\frac{p}{q}} $$
 integrating both sides of the last equality with respect to $dt$ yields
$$\int_{\mathbb{R}^2}\left(\int_{\mathbb{R}^2}|W_{f,g}^{A_1,A_2}(t,u)|_{Q}^{p}du\right)dt\leq \frac{|b_1 b_2|^{\frac{-p}{2}+1}}{(2\pi)^{p}}\int_{\mathbb{R}^2}\left(\int_{\mathbb{R}^2}|f(t+\frac{s}{2})\overline{g}(t-\frac{s}{2})|_{Q}^{q}ds\right)^{\frac{p}{q}}dt .$$

Using relation (3.3) in the proof of theorem 1 in \cite{Elliot}, we have
$$\int_{\mathbb{R}^2}\left(\int_{\mathbb{R}^2}|f(t+\frac{s}{2})\overline{g}(t-\frac{s}{2})|_{Q}^{q}ds\right)^{\frac{p}{q}}dt \le C [|f|_{2,Q}|g|_{2,Q}]^p,$$
where $C$ is a positive constant.\\
Consequently, we obtain

$$\int_{\mathbb{R}^2}\int_{\mathbb{R}^2}|W_{f,g}^{A_1,A_2}(t,u)|_{Q}^{p}dudt\leq C \frac{|b_1 b_2|^{\frac{-p}{2}+1}}{(2\pi)^{p}} |f|^p_{2,Q}|g|^p_{2,Q}.$$

\end{proof}

\section{Conclusion}
Firstly, we introduced an extension of the Winger-Ville distribution to the quaternion algebra by means of the quaternionic offset linar canonical Fourier transform (QOLCT), namely the WVD-QOLCT transform. Secondly, the Plancherel theorem and the inversion formula have been demonstrated.
Thirdly, Heisenberg's uncertainty principle and Poisson summation formula associated with WVD-QOLCT were established by using the theorems obtained for the QFT and QOLCT.
Finally the Lieb's theorem related to the WVD-QOLCT transform was formulated by applying the Lieb's theorem for the QLCT.

\end{document}